\documentclass{amsart}
\usepackage{amsthm,amsfonts,amsmath,amscd,amssymb,latexsym,epsfig}

\newcommand{\su}{{\mathfrak s  \mathfrak u}}

\newcommand{\g}{{\mathfrak g}}         

\newcommand{\cx}{{\mathbb C}}
\newcommand{\diag}{\operatorname{diag}}

\newcommand{\tr}{\operatorname{tr}}
\newcommand{\Res}{\operatorname{Res}}

\newcommand{\supp}{\operatorname{supp}}

\numberwithin{equation}{section}

\newtheorem{theorem}{Theorem}[section]

\newtheorem{lemma}[theorem]{Lemma}

\newtheorem{proposition}[theorem]{Proposition}

\theoremstyle{remark}

\newtheorem{remark}[theorem]{Remark}

\newcommand{\oP}{{\mathbb{P}}}

\newcommand{\oR}{{\mathbb{R}}}

\newcommand{\oZ}{{\mathbb{Z}}}

\newcommand{\sA}{{\mathcal{A}}}   

\newcommand{\sG}{{\mathcal{G}}}   

\newcommand{\sO}{{\mathcal{O}}}

\newcommand{\fG}{{\mathfrak{g}}}
\newcommand{\fH}{{\mathfrak{h}}}

\newcommand{\fK}{{\mathfrak{k}}}
\newcommand{\fL}{{\mathfrak{l}}}

\newcommand{\fO}{{\mathfrak{o}}}
\newcommand{\fP}{{\mathfrak{p}}}

\newcommand{\fS}{{\mathfrak{s}}}

\newcommand{\fU}{{\mathfrak{u}}}

\begin{document}

\title{A hyperk\"ahler submanifold of the monopole moduli space}
\author{Roger Bielawski}
\address{Institut f\"ur Differentialgeometrie\\
Universit\"at Hannover\\ Welfengarten 1\\ D-30167 Hannover}


\begin{abstract} We discuss a $4[k/2]$-dimensional complete hyperk\"ahler submanifold of the $(4k-4)$-dimensional moduli space of strongly centred $SU(2)$-monopoles of charge $k$.
\end{abstract}

\maketitle

\thispagestyle{empty}

The isometry group of $\oR^3$ acts isometrically on the moduli space $M_k$ of Euclidean $SU(2)$-monopoles of charge $k$ and, consequently, a fixed point set of any subgroup of this group is a totally geodesic submanifold. Similar statement holds for any subgroup of the orthogonal group $O(3)$ acting on the submanifold $M_k^0$ consisting of {\em strongly centred} monopoles \cite{HMM} (these are monopoles with the centre at the origin and total phase equal to $1$). Houghton and Sutcliffe \cite{HS} have shown that the submanifold of $M^0_3$ consisting of monopoles symmetric about the origin is isometric to the Atiyah-Hitchin manifold (i.e. the moduli space of centred monopoles of charge $2$). Surprisingly, monopoles of higher charges invariant under the reflection $x\mapsto -x$ seem not to have been considered in the literature. This reflection is particularly interesting since it extends to a reflection $\tau:(x,t)\mapsto (-x,t^{-1})$ on $\oR^3\times S^1$ which preserves the hyperk\"ahler structure of $\oR^3\times S^1$. It is then easy to deduce that the submanifold $N_k$ of strongly centred monopoles symmetric about the origin is a (complete) hyperk\"ahler submanifold of $M_k^0$ for any charge $k$. We have stumbled upon this hyperk\"ahler manifold (for even $k$) in a completely different context in \cite{slices} and realised only {\em a posteriori}, by identifying the twistor space, that it must be a submanifold of $M_k^0$.
\par
In the present paper we describe the submanifold $N_k$ in terms of Nahm's equations. Since $N_k$ is $SO(3)$-invariant, all of its complex structures are equivalent and can be identified with a complex submanifold of based rational maps of degree $k$. This is straightforward, given that the involution $\tau$ acts on rational maps via $p(z)/q(z)\mapsto \tilde{p}(z)/q(-z)$, where $\tilde p(-z)p(z)-1=0 \mod q(z)$, but we also show this directly using Nahm's equations. We then show that $N_k$ is biholomorphic to the {\em transverse Hilbert scheme} of $n$ points \cite[\S 5]{slices} on the $D_1$-surface if $k=2n$, and on the $D_0$-surface if $k=2n+1$. This allows us to conclude that $N_{2n}$ is simply connected, while $N_{2n+1}$ has fundamental group of order $2$. In Section \ref{2} we present an alternative construction, which also gives a description of hyperk\"ahler deformations of $N_{2n}$.

\section{Description in terms of Nahm's equations}

The moduli space $M_k^0$ of strongly centred $SU(2)$-monopoles of charge $k$ is isomorphic to the moduli space of $\su(k)$-valued solutions to Nahm's equations on $(0,2)$ with $T_1(t),T_2(t),T_3(t)$ having simple poles at $t=0,2$, the residues of which define the standard $k$-dimensional irreducible representation of $\su(2)$, i.e.
$$\Res T_1(t)=i\diag(k-1,k-3,\dots,3-k,1-k),$$
$$\Res (T_2+iT_3)(t)_{ij}=\begin{cases}\sqrt{j(k-j)} & \text{if $i=j+1$}\\ 0 & \text{otherwise}.\end{cases}$$
The corresponding representation of $\fS\fL_2(\cx)$ is given by the action $y\partial_x,x\partial_y,x\partial_x-y\partial_y$ on binary forms of degree $k-1$ with a basis
\begin{equation}v_i=\begin{pmatrix} k-1\\i-1\end{pmatrix}^{1/2}x^{k-i}y^{i-1},\enskip i=1,\dots,k.\label{basis}\end{equation}
Write $V$ for the standard $2$-dimensional $\fS\fL_2(\cx)$-module, so that the symmetric product $S^{k-1}V$ is the standard $k$-dimensional irreducible representation of $\fS\fL_2(\cx)$.
The equivariant isomorphisms
$$ \Lambda^2\bigl(S^{2n+1}V^\ast\bigr)\simeq \bigoplus_{i=0}^{n} S^{4i}V^\ast,\quad S^2 \bigl(S^{2n}V^\ast\bigr)\simeq\bigoplus_{i=0}^{n} S^{4i}V^\ast,$$
imply that, if $k$ is even (resp. if $k$ is odd), then there exists a unique (up to scaling) $\fS\fL_2(\cx)$-invariant skew-symmetric (resp. symmetric) bilinear form on $S^{k-1}V$.
This invariant form is a classical object and is called {\em transvectant}. In 
the basis \eqref{basis} the transvectant $T(f,g)$ of $f=\sum_{i=1}^k a_iv_i$, $g=\sum_{i=1}^k b_iv_i$ is given \cite[Ex.2.7]{SJ}
$$T(f,g)=\sum_{i=0}^{k-1} (-1)^{i}a_{i+1}b_{k-i}.$$
We conclude that the residues of $T_1,T_2,T_3$  belong to a symplectic subalgebra of $\su(k)$ if $k$ is even, and to an orthogonal subalgebra if $k$ is odd. These subalgebras are defined as
\begin{equation} \left\{ A\in \su(k);\; AJ+JA^T=0\right\},\label{subgroup}\end{equation}
where $J$ is an antidiagonal matrix  with $J_{i,k+1-i}=(-1)^{i-1}$. In other words, they are the fixed point sets of the involution 
\begin{equation} \sigma(A)= -JA^TJ^{-1}. \label{tau}\end{equation}
We denote by $\su(k)^\sigma$ the subalgebra \eqref{subgroup} and by  $SU(k)^\sigma$ the corresponding subgroup of $SU(k)$ ($SU(k)^\sigma\simeq Sp(n)$ if $k=2n$ and $SU(k)^\sigma\simeq SO(2n+1)$ if $k=2n+1$). 
\par
We consider the space $\sA^\sigma$  of $\su(k)^\sigma$-valued solutions to Nahm's equations on $[0,2]$. Let $\sG$ denote the group of $SU(k)$-valued gauge transformations which are identity at $t=0,2$, and let $\sG^\sigma$ be its subgroup of $SU(k)^\sigma$-valued gauge transformations. It is easy to verify that two  $\sG$-equivalent elements of $\sA^\sigma$ are also $\sG^\sigma$-equivalent. Thus the natural map $\sA^\sigma/\sG^\sigma\to M_k^0$ is an embedding and we view $N_k=\sA^\sigma/\sG^\sigma$ as a submanifold of  $M_k^0$. $N_k$ is 
the fixed point set of an involution $\sigma$ which sends each $T_i(t)$ to $\sigma(T_i(t))$ (and acts the same way on gauge transformations) and therefore a complete hyperk\"ahler submanifold of $M_k^0$.

\begin{proposition} With respect to any complex structure $N_k$ is biholomorphic to the space of based rational maps $\frac{p(z)}{q(z)}$ of degree $k$ such that
\begin{itemize}
\item[(i)] if $k=2n$, then $q(z)=\tilde q(z^2)$ for a monic polynomial $\tilde q$ of degree $n$ and $p(z)p(-z)\equiv 1\mod q(z)$;
\item[(ii)] if $k=2n+1$, then $q(z)=z\tilde q(z^2)$ for a monic polynomial $\tilde q$ of degree $n$, $p(0)=1$ and $p(z)p(-z)\equiv 1\mod q(z)$. 
\end{itemize}
In particular $\dim_{\oR} N_{2n}=\dim_{\oR} N_{2n+1}=4n$.
\label{hol}\end{proposition}
\begin{remark} If the roots of $q(z)$ are distinct, then the above condition means that $p(w)p(-w)=1$ for each root $w$ of $q(z)$ (and $p(0)=1$ if $k$ is odd). The full $N_k$ is then the closure of this set inside the space of all rational maps.
\end{remark}
\begin{proof}
Recall \cite{Don} that $M_k$ is biholomorphic to the space of based (i.e. $f(\infty)=0$) rational maps of degree $k$, and  $M_k^0$ to the submanifold consisting of rational maps $p(z)/q(z)$ such that the sum of the poles $z_i$ is equal to $0$ and $\prod_{i=1}^k p(z_i)=1$. From the point of view of Nahm's equations, the rational map is obtained by applying a (singular) complex gauge transformation $g$ to the complex Nahm equation $\dot\beta=[\beta,\alpha]$ in order to make $\beta$ a constant matrix of the form
\begin{equation}S=\begin{pmatrix} 0 & \dots & \dots & 0  & s_n\\ 1 & \ddots & & 0 & s_{n-1}\\ \vdots & \ddots &\ddots & \vdots &\vdots\\ 0 &  \dots &\ddots & 0&  s_2\\  0 & \dots &\dots & 1 & s_1\end{pmatrix}.\label{S}\end{equation}
The value of the complex gauge transformation $g(t)$ at $t=2$ (modulo a fixed singular gauge transformation) is then an element $u$ of the centraliser of $S$ in $GL(k,\cx)$. The pair $(S,u)$ corresponds to the rational map $\tr u(z-S)^{-1}$ (this is the description given in \cite{JLMS}). The complex structures of $M_k^0$ and of $N_k$ are obtained by assuming that $S$ and $u$ belong to the appropriate subalgebra and subgroup (i.e. to $\fS\fL(k,\cx),SL(n,\cx)$ for $M_k^0$ and to $\fS\fL(k,\cx)^\sigma,SL(n,\cx)^\sigma$ for $N_k$). It is enough to consider the subset where the poles of the rational map, i.e. the eigenvalues of $S$, are distinct (since $N_k$ is the closure of this set in the space of all rational maps). A Cartan subalgebra $\fH$ of  $\fS\fL(k,\cx)^\sigma$ is given by the diagonal matrices $h$ satisfying $ h_{ii}+h_{k+1-i,k+1-i}=0$, $i=1,\dots,k$. It is then immediate that if $(S,u)$ is conjugate to an element of $\fH\times H$ ($H=\exp\fH)$, then the corresponding rational map satisfies the 
conditions in the statement.
\end{proof}

In \cite[Ex. 5.4]{slices} we have identified the complex manifold described in the above proposition for $k=2n$ as the {\em Hilbert scheme of $n$ points on the $D_1$-surface $x^2-zy^2=1$ transverse to the projection $(x,y,z)\mapsto z$} (similarly, the space of all rational maps of degree $k$ is the Hilbert scheme of $k$ points on $\cx^\ast\times \cx$ transverse to the projection onto the second factor \cite[Ch. 6]{AH}).
It turns out that for odd $k$ the complex structure of $N_k$ is that of the transverse Hilbert scheme of points on the $D_0$-surface $x^2-zy^2-y=0$:
\begin{proposition}\begin{itemize}
\item[(i)] With respect to any complex structure $N_{2n}$ is biholomorphic to the Hilbert scheme of $n$ points  on the $D_1$-surface $x^2-zy^2=1$ transverse to the projection $(x,y,z)\mapsto z$.
\item[(ii)] With respect to any complex structure $N_{2n+1}$ is biholomorphic to the Hilbert scheme of $n$ points  on the $D_0$-surface $x^2-zy^2-y=0$ transverse to the projection $(x,y,z)\mapsto z$.
\end{itemize}
\end{proposition}
\begin{remark} The fact that $N_3$ is biholomorphic to the $D_0$-surface has been observed by Houghton and Sutcliffe \cite{HS}.
\end{remark}
\begin{proof} Part (i) has already been shown in  \cite[Ex. 5.4]{slices}. For part (ii) recall from \cite{slices} that the Hilbert scheme of $n$ points  on the $D_0$-surface $x^2-zy^2+y=0$ transverse to the projection $(x,y,z)\mapsto z$ is an affine variety in $\cx^{3n}$ given by the same equation, but for polynomials. More precisely its points are polynomials $x(z),y(z),r(z)$ with degrees of $x$ and $y$ at most $n-1$ and $r(z)$ a monic polynomial of degree $n$, satisfying the condition
$$ x(z)^2-zy(z)^2-y(z)=0\mod r(z).$$
As in \cite{slices} write $z=u^2$ and rewrite the above equation as
$$ \bigl(x(u^2)+uy(u^2)\bigr)\bigl(x(u^2)-uy(u^2)\bigr)- \frac{ \bigl(x(u^2)+uy(u^2)\bigr)-\bigl(x(u^2)-uy(u^2)\bigr)}{u}=0\mod r(u^2).$$
Let us write $f(u)=x(u^2)+uy(u^2)$, $p(u)=1+uf(u)$, $q(u)=ur(u^2)$. Then it is easy to see that if the roots of $q$ are distinct, then the last equation is equivalent to the condition $p(w)p(-w)=1$ for any nonzero root $w$ of $q(u)$.  Since this last equation is polynomial in the coefficients of $p$ and $q$, it describes a closed affine subvariety inside the affine variety of all rational maps. Since the two closed affine subvarieties, namely the transverse Hilbert scheme of points on the $D_0$-surface and the variety described in Proposition \ref{hol}, have a common open dense subset, they must coincide.
\end{proof}

We can now compute the fundamental group of $N_k$:
\begin{proposition} $$\pi_1(N_k)=\begin{cases} 1 & \text{if $k$ is even,}\\ \oZ_2 & \text{if $k$ is odd}.\end{cases}$$
\label{pi1}\end{proposition}
\begin{proof} The fundamental groups of the $D_0$- and $D_1$-surface (i.e. the Atiyah-Hitchin manifold and its double cover) are well-known \cite{AH} and equal to $\oZ_2$ and to $1$ respectively. The result follows from
\begin{lemma} Let $X$ be a smooth complex surface and $\pi:X\to C$ a holomorphic submersion onto a connected Riemann surface $C$. Suppose further that, over an open dense subset of $C$, $\pi$ is a locally trivial fibration with connected fibres. Then the fundamental group of the transverse Hilbert scheme $X^{[n]}_\pi$ of $n$ points, $n\geq 2$, is equal to $H_1(X,\oZ)$.
\end{lemma}
We first observe that if $X$ is a smooth complex surface and $n\geq 2$, then the fundamental group of the full Hilbert scheme $X^{[n]}$ of $n$ points, $n\geq 2$, is equal to $H_1(X,\oZ)$. This follows by combining two facts: 1) a classical result of Dold and Puppe \cite[Thm. 12.15]{DP} which says that $\pi_1(S^n X)\simeq H_1(X,\oZ)$ for $n\geq 2$, and 2) a theorem of Koll\'ar \cite[Thm. 7.8]{Koll} which implies that $\pi_1(X^{[n]})\simeq\pi_1(S^n X)$. 
\par
We now aim to show that $\pi_1\bigl(X^{[n]}_\pi\bigr)\simeq\pi_1\bigl(X^{[n]}\bigr)$. Let $Y$ denote a submanifold of $X^{[n]}$ consisting of $D\in X^{[n]}$ with $\pi(D)=z_0+ E$, where $E$ has length $n-2$ in $C$, $z_0\not\in \supp E$, and $D\cap \pi^{-1}(z_0)$ consists of two distinct points. Clearly $Y\cap X^{[n]}_\pi=\emptyset$. Denote by $U$ a tubular neigbourhood of $Y$ and let $Z=U\cup X^{[n]}_\pi$, $W=U\cap X^{[n]}_\pi$. Since the complement  of $Z$ in $X^{[n]}$ has (complex) codimension $2$, the fundamental groups of  $Z$ and of $X^{[n]}$ coincide. Since $W$ is a punctured disc bundle over $Y$, the long exact sequence of homotopy groups implies that the map $\pi_1(W)\to\pi_1(U)$ is surjective and its kernel consists of at most the ``meridian loop"  around $Y$. Owing to the assumption, we can choose a point $y$ of $Y$ so that a neighbourhood of the fibre containing two points is isomorphic to $B\times F$, where $B$ is a disc $\{z\in\cx; |z|< 1+\epsilon\}$ and $F$ is the generic fibre of $\pi$. The intersection of a neighbourhood of $y$ in $Z$ with $W$  is then of the form $S^2_\pi(B\times F)\times V$, where
the subscript $\pi$ means that the pair of points $\{x_1,x_2\}\in B\times F$ satisfies $\pi(x_1)\neq \pi(x_2)$ and  $V$ is an open subset in $X^{[n-2]}_\pi$. Let $\rho$ be the projection $(B\times F)\times (B\times F)\to S^2(B\times F)$.
A ``meridian loop" in $W$ around $Y$ can be then chosen to be 
$$t\mapsto \bigl(\rho(e^{\pi it},f,e^{-\pi i t},f),v\bigr),\enskip t\in[0,1]$$
for constant $f$ and $v$. This loop is contractible in $X^{[n]}_\pi$: the homotopy
$$ H(r,t)= \bigl(\rho(re^{\pi it},f,re^{-\pi i t},f),v\bigr),\enskip t,r\in[0,1]$$
contracts it to $(D,v)$ where $D$ is the double point $(\{z^2=0\},f)$ in $B\times F$. 
To recapitulate: we have shown that the map $\pi_1(W)\to \pi_1(U)$ is surjective and its kernel has trivial image in $\pi_1(X^{[n]}_\pi)$. It  follows that the amalgamated free product     
$\pi_1(U)\ast_{\pi_1(W)}\pi_1(X^{[n]}_\pi)$ is isomorphic to $\pi_1(X^{[n]}_\pi)$ and, hence,  van Kampen's theorem implies that $\pi_1(Z)\simeq \pi_1(X^{[n]}_\pi)$.
\end{proof}

\begin{remark} The $D_0$-surface $X$ is the quotient of the $D_1$-surface $\tilde X$ by a free action of $\oZ_2$ given by $(x,y,z)\mapsto (-x,-y,z)$. This induces a free $\oZ_2$-action on the transverse Hilbert scheme $\tilde X_\pi^{[n]}$ of $n$ points for any $n$, but $\tilde X_\pi^{[n]}/\oZ_2\not\simeq X_\pi^{[n]}$, unless $n=1$. Certainly, there is a surjective holomorphic map $\tilde X_\pi^{[n]}\to X_\pi^{[n]}$ for any $n$, which, in the description of these spaces provided in Proposition \ref{hol}, sends a rational function $p(z)/q(z)\in \tilde X_\pi^{[n]}$ to $\bar{p}(z)/zq(z)$, where $\bar{p}(z)=p(z)^2\mod zq(z)$. This map is constant on $\oZ_2$-orbits, but it is
generically $2^n$-to-$1$ and not a covering (the preimage of a point consists of $2^m$ points, where $2m$ is the number of distint roots of $q(z)$).
\label{Z2}\end{remark}
\begin{remark} It is instructive to compare spectral curves of monopoles in $N_{2n+1}$ to those in $N_{2n}$. It follows from Proposition \ref{hol} that the spectral curve $S$ of a monopole in $N_{2n+1}$ is always singular and given by an equation of the form $\eta P(\zeta,\eta)=0$, where $\zeta$ is the affine coordinate of $\oP^1$, $\eta$ is the induced fibre coordinate in $T\oP^1$, and $P$ a  polynomial of the form $\eta^{2n}+\sum_{i=1}^{n}a_i(\zeta)\eta^{2n-2i}$, $\deg a_i(\zeta)=4i$. A spectral curve of a monopole satisfies \cite{AH} the condition $L^2|_S\simeq \sO$, where $L^2$ is a line bundle on $T\oP^1$ with transition function $\exp(2\eta/\zeta)$. It follows that the line bundle $L^2$ is also trivial on the curve $\tilde S$ defined by the equation $P(\zeta,\eta)=0$. Conversely, the spectral curve $\tilde S$ of a monopole in $N_{2n}$ admits a section $s(\zeta,\eta)$ of $L^2$ which satisfies $s(\zeta,\eta)s(\zeta,-\eta)\equiv 1\mod P(\zeta,\eta)$, where $P(\zeta,\eta)=0$ is the equation of $\tilde S$. It follows that $s(\zeta,0)=\pm 1$ and if we set $\bar{s}(\zeta,\eta)=s^2(\zeta,\eta)$ on $\tilde S$ and $\bar{s}(\zeta,\eta)\equiv 1$ on $\eta=0$, we obtain a nonvanishing section of $L^4$ on the curve $\eta P(\eta,\zeta)=0$, i.e. a section of $L^2$ on the curve $S$ given by $\tilde \eta P(\tilde\eta/2,\zeta)=0$, where $\tilde\eta=2\eta$. In the case $n=1$, Houghton and Sutcliffe \cite{HS} have shown that for $n=1$ these maps $S\mapsto \tilde S$ and $\tilde S\mapsto S$ send spectral curves of monopoles  in $N_{3}$ to spectral curves of monopoles in $N_{2}$ and vice versa\footnote{Strictly speaking, the curve $\tilde S$ obtained from $S$ must be rescaled via $\eta=2\tilde\eta$ in order to be the spectral curve of a monopole in $N_{2}$.}, but for higher $n$ this is not  the case. The reason is that Hitchin's \cite{Hit} nonsingularity condition $H^0(S,L^t(k-2))=0$, $t\in (0,2)$, is not necessarily satisfied for the resulting curves.
\end{remark}

\section{Deformations and coverings\label{2}}

Dancer \cite{Dan} has shown  that the $D_1$-surface admits a $1$-parameter family of deformations carrying complete hyperk\"ahler metrics. As we observed in \cite{slices}, the transverse Hilbert schemes of points on these deformations also admit natural complete hyperk\"ahler metrics. We wish to describe these metrics as deformations of manifolds $N_{2n}$. We begin by describing $N_k$ without reference to an embedding into $M_k^0$.
\par
Let $G_k$ (resp. $\g_k$) denote $Sp(n)$ (resp. $\fS\fP(n)$) if $k=2n$ and $SO(2n+1)$ (resp. $\fS\fO(k)$) if $k=2n+1$.
The construction of the previous section shows that $N_k$ is the moduli space of $\fG_k$-valued solutions to Nahm's equations on $(0,2)$ with simple poles at $t=0,2$ and residues defining the principal homomorphism $\su(2)\to \g_k$, modulo $G_k$-valued gauge transformations which are identity at $t=0,2$. This moduli space is, in turn, a finite-dimensional hyperk\"ahler quotient of a simpler hyperk\"ahler manifold.
Let $W_k^-$ (resp. $W_k^+$)  be the moduli space of  $\fG_k$-valued solutions to Nahm's equations on $(0,1]$ (resp. $[1,2)$) with the above boundary behaviour at $t=0$ (resp. at $t=2$) and regular at $t=1$, modulo $G_k$-valued gauge transformations which are identity at $t=0,1$ (resp. at $t=1,2)$. $W_k^\pm$ are hyperk\"ahler manifolds (biholomorphic to $G_k^\cx\times \cx^n$ \cite{JLMS}) with an isometric and triholomorphic action of $G_k$ obtained by allowing gauge transformations with an arbitrary value at $t=1$. Then $N_k$ is the hyperk\"ahler quotient of $W_k^-\times W_k^+$ by the diagonal $G_k$.
\par
We can also describe in a similar manner the universal (i.e. double) covering space of $N_{2n+1}$: it is given by the same construction, but with $G_{2n+1}=\text{\it Spin}(2n+1)$ instead of $SO(2n+1)$. 
\par
An alternative construction of the $N_k$ proceeds as follows. Let $G$ denote one of the groups $Sp(n)$, $SO(2n+1)$, or $\text{\it Spin}(2n+1)$ and let $\tau$ be an automorphic involution on $G$ with fixed point set $K$. Consider the hyperk\"ahler quotient $Y^-$ of $W_k^-$ by $K$ (with zero-level set of the moment map). Let $(T_0,T_1,T_2,T_3)$ be  a solution to Nahm's equations corresponding to a point in $Y^-$. Modulo gauge transformations we can assume that $T_0(1)=0$. We can then extend this solution to a solution to Nahm's equations on $(0,2)$ by setting 
\begin{equation} T_i(2-t)=-\tau\bigl(T_i(t)),\enskip i=0,1,2,3.\label{tau2}\end{equation}
This solution has the boundary behaviour of a solution in $N_k$ and we can describe the moduli space $Y$ of such extended solutions as the space of solutions on $(0,2)$ having the correct poles and residues at $t=0,2$ and satisfying \eqref{tau2}, modulo $G$-valued gauge transformations $g(t)$ such that
\begin{equation} g(0)=g(2)=1,\quad g(2-t)=\tau(g(t)), \enskip t\in [0,2].\end{equation}
The map $Y^-\to Y$ is a triholomorphic homothety with factor $2$.  For dimensional reasons $Y^-$ (and consequently $Y$) is empty unless $\fK=\fU(n)$ for $k=2n$ or $\fK=\fS\fO(n)\oplus \fS\fO(n+1)$ for $k=2n+1$. Thus there are the following three possibilities for the symmetric pair $(G,K)$:
\begin{itemize}
\item[(i)] $k=2n$, $G=Sp(n)$ and $K=U(n)$;
\item[(ii)] $k=2n+1$, $G=SO(2n+1)$ and $K=S\bigl(O(n)\times O(n+1))$;
\item[(iii)] $k=2n+1$, $G=\text{\it Spin}(2n+1)$ and $K$ is the diagonal double cover of $SO(n)\times SO(n+1)$ (i.e. $K=\text{\it Spin}(n)\times \text{\it Spin}(n+1)/\{(1,1),(-1,-1)\}$).
\end{itemize}
An easy computation shows that in each case $\dim Y=\dim N_k$. Moreover, the natural map $Y\to N_k$ is an isometric (and triholomorphic) immersion, and since both $N_k$ and $Y$ are complete (\cite[Thm. A.1]{slices}), this map must be a covering. Thus it follows from Proposition \ref{pi1} that $Y$ is isometric to $N_k$ in cases (i) and (ii), while in case (iii) $Y$ is the universal cover of $N_{2n+1}$.
\par
The above construction allows us easily to describe a family of deformations of $N_{2n}$. Indeed, the
Lie algebra $\fK=\fU(n)$ has a nontrivial centre and, therefore, we can take hyperk\"ahler quotients of $W_{2n}^-$ by $K$ at nonzero level sets of the hyperk\"ahler moment map. This produces a $3$-parameter family of hyperk\"ahler deformations of $N_{2n}$. Arguments analogous to those in \cite{slices} show that these are the natural hyperk\"ahler metrics on the transverse Hilbert schemes of points on Dancer's deformations of the $D_1$-surface. 

\begin{remark} As already mentioned in Remark \ref{Z2}, $N_{2n}$ admits a free action of $\oZ_2$ for any $n$. This action is also isometric and triholomorphic and, hence, $N_{2n}/\oZ_2$ is a hyperk\"ahler manifold. This manifold can be described in the same way as $N_{2n}$ but with $G_{2n}=\oP Sp(n)$ rather than $Sp(n)$.  
\end{remark}

\end{document}